\documentclass[12pt]{amsart}
\usepackage[bookmarks=true,colorlinks=true,  % comment this out 
pdfstartview=FitV, linkcolor=blue, citecolor=red,% for editing
urlcolor=green]{hyperref}
\usepackage{amsmath,amsfonts,amssymb}
\usepackage[mathscr]{eucal}  % for script fonts
\usepackage{mathtools}% for xmapsto
\usepackage{hyperref}
\usepackage{xypic} % for commutative diagram
\title[Affine McShane Identities]{McShane-type identities for affine deformations}

\author[Charette]{Virginie Charette}
    \address{{\it Charette:\/}
    D\'epartement de math\'ematiques\\ Universit\'e de Sherbrooke\\
Sherbrooke,  Qu\'ebec J1K 2R1  Canada}
    \email{v.charette@usherbrooke.ca}
\author[Goldman]{William M. Goldman}
    \address{{\it Goldman:\/}
    Department of Mathematics\\ University of Maryland\\
    College Park, MD 20742 USA}
    \email{wmg@math.umd.edu}

\date{\today}

\thanks{Charette gratefully acknowledges partial support from the
Natural Sciences and Engineering Research Council of Canada.
Goldman gratefully acknowledges partial support from National
Science Foundation grant DMS070781. 
Both authors acknowledge support from U.S. National Science Foundation grants DMS 1107452, 1107263, 1107367 ``RNMS: Geometric structures And Representation varieties" (the GEAR Network) to facilitate this research.}

\newtheorem{thm}{Theorem}
\newtheorem{lem}[thm]{Lemma}
\newtheorem{corollary}[thm]{Corollary}


\begin{document}

\begin{abstract}
%The purpose of this note is to 
We derive an identity for Margulis invariants
of affine deformations of a complete orientable one-ended hyperbolic surface
following the identities of McShane, Mirzakhani and  Tan-Wong-Zhang.
As a corollary, a deformation of the surface which 
infinitesimally lengthens all interior simple closed curves must
 infinitesimally lengthen the boundary.
\end{abstract}
\maketitle
%\tableofcontents

%\newtheorem*{thm*}{Theorem}
%\newtheorem{lem}[thm]{Lemma}
%\newtheorem{corollary}[thm]{Corollary}
\newtheorem{prop}[thm]{Proposition}

\newcommand{\R}{\mathbb{R}} % real numbers
\newcommand{\Z}{\mathbb{Z}}  % rational integers
\newcommand{\N}{\mathbb{N}}  % positive integers
\newcommand{\A}{\mathbb{A}}  % affine space
\newcommand{\DD}{\mathscr{D}}  % MIrzakhani's D
\newcommand{\RR}{\mathscr{R}}  % MIrzakhani's R
\newcommand{\EE}{\mathscr{E}}  % Norbury's E

\newcommand{\FO}{\mathcal{C}} % pairs of ssc; first case
\newcommand{\FS}[1]{\mathcal{C}_{#1}'} % ssc; second case
\newcommand{\FU}{\mathcal{C}''} % ssc; third case

\newcommand{\FN}[1]{\FO_{#1}}  % curves of length between n and n+1

\newcommand{\SL}{\mathsf{SL}} % special linear group
\newcommand{\SOto}{\mathsf{SO}(2,1)} % special orthogonal group
\newcommand{\PSLtR}{\SOto} % do we need both groups?
\newcommand{\Ht}{\mathsf{H}^2} % hyperbolic plane
\newcommand{\Eto}{\mathsf{E}^{2,1}} % Minkowski 2+1-space
\newcommand{\Rto}{\R^{2,1}} % Minkowski 2+1-space
\newcommand{\Isom}{\mathsf{Isom}} % isometries 
\newcommand{\Iso}{\Isom^+(\Eto)} % orientation-preserving isometries of Minkowski space
\newcommand{\Hom}{\mathsf{Hom}} % isometries 
\newcommand{\interior}{\mathsf{interior}} % interior
\newcommand{\gammas}{\{\gamma_1,\gamma_2\}} % unordered pair of two scc 
\newcommand{\OB}{\mathsf{O}} % summation for one boundary component
\newcommand{\SB}{\mathsf{S}} % summation for several boundary components
\newcommand{\NO}{\mathsf{N}} % summation for nonorientable surfaces

\newcommand{\Fricke}{{\mathfrak F}} % Fricke space
\newcommand{\RPtwo}{{\R\mathsf{P}^2}} % RP2
\newcommand{\Gg}{\mathfrak{g}} % Lie algebra of G
\newcommand{\Ad}{\mathsf{Ad}} % adjoint representation
\newcommand{\Aut}{\mathsf{Aut}} % automorphism group
\newcommand{\vzero}{\mathbf{0}} % zero vector, zero section
\newcommand{\GG}{\mathbf{G}} % affine group scheme
\newcommand{\Ree}{\R[\epsilon]} % dual numbers

\section*{Introduction}
Beginning with McShane~\cite{MR1625712}, 
sums of geometric quantities over the simple closed curves on a surface $\Sigma$
have been shown to exhibit remarkable properties, 
such as being constant over the deformation space of geometric structures on $\Sigma$. 
Let $\Fricke(\Sigma)$ denote the {\em Fricke space,\/} consisting of isotopy classes of marked
complete hyperbolic structures on $\Sigma$. 
If $\gamma$ is a closed curve on a marked hyperbolic surface $X$,
then the function 
\begin{align*}
\Fricke(\Sigma) &\longrightarrow \R \\
X &\longmapsto \ell_X(\gamma) \end{align*}
assigning to $X$ the length of the closed geodesic on $X$ 
(or $0$ if $\gamma$ is homotopic to a cusp) is a basic quantity, 
depending on the homotopy class of $\gamma$.
See Abikoff~\cite{MR590044} and Wolpert~\cite{MR2641916}, \S 3  for background on hyperbolic Riemann surfaces,
where $\Fricke(\Sigma)$ idenfifies with the {\em Teichm\"uller space\/} of $\Sigma$
by the uniformization theorem.

Mirzakhani~\cite{MR2264808} and
Tan-Wong-Zhang~\cite{MR2215456} express the length of a boundary component as an infinite series of functions of the lengths of interior curves. 
(Compare also Hu-Tan-Zhang~\cite{HTZ} and 
Wolpert~\cite{MR2641916}, pp.\ 86--87.)
This led Mirzakhani to new calculations of the 
Weil-Petersson volume of the moduli space of curves. 

In another direction,  
Labourie-McShane~\cite{MR2541705} 
found identities involving quantities generalizing  geodesic lengths for
% {\em higher Teichm\"uller spaces\/} 
{\em Hitchin representations\/} of $\pi_1(\Sigma)$ into $\SL(n,\R)$. 
For $n=3$, these quantities are the geodesic length functions for the 
Hilbert metric on convex $\RPtwo$-structures on $\Sigma$.
In contrast, our identities are the first such identities for a 
% {\em mixed type\/} Lie groups.
 {\em non-reductive\/} Lie group. 

Closely related to hyperbolic structures on sufraces are 
{\em flat Lorentzian structures\/} on $3$-manifolds,
that is, geometric structures modeled on Minkowski space 
$\Eto$ and groups of (Lorentzian) isometries.  
A {\em Margulis spacetime\/} $M^3$ is a quotient $\Eto/\Gamma$, 
where $\Gamma\subset\Isom(\Eto)$ is a discrete group of
Lorentzian isometries which acts properly and freely on $M$. 
We assume $\Gamma$ is not solvable, as the solvable cases are easily classified
(Fried-Goldman~\cite{MR689763}).
Then $M^3$ is an geodesically complete flat Lorentzian manifold with free fundamental group $\Gamma$.
It is necessarily orientable (%Charette-Drumm-Goldman~
\cite{MR3180618}).
The linearization $\Isom^+(\Eto) \longrightarrow \SOto$ maps $\Gamma$
isomorphically onto a discrete subgroup of $SOto\cong\Isom(\Ht)$ (\cite{MR689763}).
Therefore, associated to $M^3$ is a complete hyperbolic surface $X = \Ht/\Gamma_0$.
Just as %nonperipheral 
essential closed curves in $X$ are canonically homotopic to
unique closed geodesics (having length $\ell_X(\gamma)$),
nonparabolic closed curves in $M$ are canonically homotopic to closed geodesics
which are necessarily spacelike. These geodesics have a natural Lorentzian length $\alpha(\gamma)$,
called the {\em Margulis invariant,\/} as this quantity was 
discovered and developed by Margulis~\cite{MR741860,MR722330}.
Such a geodesic is the quotient of a line $\mathsf{Axis}(\gamma)$ upon which
$\gamma$ acts by translation of Lorentzian displacement $\alpha(\gamma)$.
This is similar to how $\ell(\gamma)$ measures the minimal displacement 
of the linear part of $\gamma$ acting isometrically on $\Ht$.
(Compare the survey articles Abels~\cite{MR1866854}, 
Charette-Drumm~\cite{MR3379819}, Drumm~\cite{MR2987620} and
Gu\'eritaud~\cite{MR3379833} for background on these geometric structures.)

This paper develops an identity for Margulis spacetimes for the Margulis invariants
$\alpha(\gamma)$ analogous to McShane's identity for hyperbolic surfaces. 

The Lie group of interest here is the {\em tangent bundle\/} Lie group $TG$ (see below)a
of the isometry group $G = \Isom(\Ht)$ of the hyperbolic plane $\Ht$.
The Lie algebra $\Gg$ of $G$ naturally identifies with the Lorentzian vector space $\Rto$, 
and $\Ad$ identifies with the isometric action of $G \cong \SOto$ on $\Rto$. 
The affine space $\Eto$ identifies with the quotient $TG/\vzero_G$ by the zero-section.

Given any Lie group $G$ with Lie algebra $\Gg$, the tangent bundle of $G$ enjoys a Lie
group structure as the semidirect product
\begin{equation}\label{eq:TangentBundleLieGroup}    TG :=  G \rtimes_\Ad \Gg \end{equation}
where $G \xrightarrow{~\Ad~}\Aut(\Gg)$ is the adjoint representation. 
Then $\Gg$ is an abelian normal subgroup, and the zero-section 
\[ G \xrightarrow{~\vzero_G~} TG\] embeds
$G$ as a subgroup of $TG$  which is {\em not\/} normal.

This can be elegantly described in terms of base change $\R \hookrightarrow\Ree$,
where $\Ree$ is the $\R$-algebra of {\em dual numbers.} 
(Here $\epsilon$ is an indeterminate with $\epsilon^2= 0$).  
Suppose $G$ is the group  $\GG(\R)$ of $\R$-points of an algebraic group scheme $\mathbf{G}$.
Then $TG = \GG(\Ree)$ and
\[
\Hom(\Gamma, TG) = \Hom(\Gamma, \GG(\R[\epsilon]) = 
\big(\Hom(\Gamma, \GG\big) (\R[\epsilon]) =  T \Hom(\Gamma,G) \]
is the Zariski tangent bundle of the representation variety $\Hom(\Gamma,G)$.
In particular homomorphisms $\Gamma \longrightarrow TG$ correspond to 
a homomorphism $\Gamma \xrightarrow{~\rho~} G$ and an {\em infinitesimal deformation\/}
of $\rho$, that is, a tangent vector in $T_\rho\big(\Hom(\Gamma,G)\big)$.

We obtain an identity for affine deformations $\Gamma$ of a Fuchsian group 
$\Gamma_0$ by differentiating the basic length identity for $\Gamma_0$. 
To simplify notation, define two functions $H(u,v)$ and $K(u,v)$:
\begin{align*}\label{eq:defHK}
H(u,v) &:= \frac1{1 + e^{(u+v)/2}}+ \frac1{1 + e^{(u-v)/2}} \\
K(u,v) &:= \frac1{1 + e^{(u+v)/2}}- \frac1{1 + e^{(u-v)/2}}  %\\ 
=  \frac{-\sinh(v/2)}{\cosh(u/2) + \cosh(v/2)}.
\end{align*}
%\big(as in Mirzakhani~\cite{MR2264808} for $H(u,v)$\big).

\begin{thm}\label{thm:OneEnd}
Let $X = \Ht/\Gamma_0$ be a complete one-ended orientable hyperbolic surface.
For any affine deformation of $\Gamma_0$, the Margulis invariants satisfy:
\begin{align*}
\bigg(
1 - \sum_{\gammas\in \FO} 
H\big(\ell_X(\gamma_1)+\ell_X(\gamma_2),\ell_X(\partial)   \big)\bigg)\; 
& \alpha(\partial)  \\ 
= 
\sum_{\gammas\in \FO} 
K\big(\ell_X(\gamma_1)+\ell_X(\gamma_2),\ell_X(\partial)\big) \ 
\big( &\alpha(\gamma_1) + \alpha(\gamma_2)\big). 
\end{align*}
Here $\FO$ denotes the set of unordered pairs of isotopy classes of
simple closed curves such that $\gamma_1,\gamma_2$ and $\partial X$
cobound an embedded $3$-holed sphere in $X$. 
\end{thm}

Theorem~\ref{thm:OneEnd} implies the following result in hyperbolic geometry\,:

\begin{corollary}\label{cor:IdentityDerived}
Suppose that $X$ is a complete one-ended orientable hyperbolic surface.
Let $X_t$ be a deformation of hyperbolic surfaces such that 
the derivative of the length satisfies \,: 
\[ \frac{d}{dt} \ell_{X_t}(\gamma) > 0 \] 
for all simple closed curves $\gamma\subset\interior(X)$.
Then 
\[ \frac{d}{dt} \ell_{X_t}(\partial) > 0. \] 
\end{corollary}

Yair  Minsky pointed out that this corollary can also be deduced 
from Bestvina-Bromberg-Fujiwara-Souto~\cite{MR3145000},
using the convexity of length functions with respect to shear deformations.
Fran\c cois  Gu\'eritaud observed that this corollary also follows from the theory of strip
deformations developed in his collaboration~\cite{DGK1,MR348055,MR3379833} 
with Jeffrey Danciger and Fanny Kassel.
In particular, he described deformations where all the interior simple loops 
lengthen, and all but one boundary component lengthens.

McShane~\cite{MR2346506} also discusses differentiated length identities and rearrangements
of absolutely convergent series. 
Papadopoulos-The\`ret~\cite{MR2587462} explains a construction of Thurston's which shortens
(or lengthens) all closed geodesics; see also \cite{DGK1,MR348055,MR3379833,GLMM}.

A direct application of the technique of Labourie-McShane~\cite{MR2541705} fails, 
since the cross-ratios of \cite{MR2541705} will not apply directly to the non-reductive group $\Isom^+(\Eto)$.
However, a generalization of these cross-ratios which includes infinitesimal information may give an alternate
proof of our results.

As noted above, McShane's original identity~\cite{MR1625712} concerned once-punctured tori.
The generalized McShane identity for a bordered {\em orientable\/} hyperbolic surface 
$X$ is due, independently, to  %McShane~\cite{MR1625712}, 
Tan-Wong-Zhang~\cite{MR2215456} and Mirzakhani~\cite{MR2264808}. 
Norbury~\cite{MR2399656} extended the identity to include non-orientable surfaces.  
In subsequent work we consider identities for affine deformations when $X$ is possibly
nonorientable and may have more than one end.

\section*{Acknowledgements}
We are grateful to S.\ Antman, 
A.\ Eskin, F.\ Gu\'eritaud, G.\ McShane,  M.\ Mirzakhani, 
Y.\ Minsky, S.P. Tan and S.\ Wolpert for helpful conversations.
We especially thank M.\ Mirzakhani for the proof of Lemma~\ref{lem:Mirzakhani}.
We thank the anonymous referee for several helpful suggestions.
%We also express our gratitude 
%%for the hospitality of the Institut Henri Poincar\'e and Institut des Hautes \'Etudes Scientifiques in 2012 when this project was begun, 
%%the Centre de Recherches Math\'ematiques in Montr\'eal in October 2012,
%%and 
%to 
%the NSF Research Network (GEAR) for facilitating visits at the University
%of Maryland and Universit\'e de Sherbrooke. 
Finally we express our appreciation to the Mathematical Sciences Research Institute for their hospitality in 2015 when this work was completed.

\section{Deformations and the Margulis invariant}

The Fricke space $\Fricke(\Sigma)$ embeds in the quotient space 
$\Hom(\Gamma,G)/G$, where $G$ acts on the $\R$-algebraic set $\Hom(\Gamma,G)$
by left-composition with inner automorphisms of $G$.
Here $G = \Isom(\Ht) \cong \SOto$, and indeed 
$\Fricke(\Sigma)$ embeds as a connected component of $\Hom(\Gamma,G)/G$.
Openness of $\Fricke(\Sigma)$ in $\Hom(\Gamma,G)/G$ was first proved by Weil~\cite{MR0137792}.
(See also Raghunathan~\cite{MR0507234}, Theorem~6.19). It is a special case of 
the openness of the holonomy map, which was first noticed by Thurston; compare the discussion 
in 
Theorem~7.2 of Kapovich~\cite{MR1792613}, 
Goldman~\cite{MR957518},
Canary-Epstein-Green~\cite{MR903850} and Koszul~\cite{MR0218485}, Chapter IV, \S3, Theorem~3.

Closedness is originally due to Chuckrow~\cite{MR0227403} 
in this particular context, but it follows from 
Kazhdan-Margulis uniform discreteness~\cite{MR0223487}; 
see Chapter VIII of ~\cite{MR0507234} , 
\S 4.12 of Kapovich~\cite{MR1792613}, 
or 
\S4.1. of Thurston~\cite{MR1435975}.
(See also Goldman-Millson~\cite{MR884798}.) 
The Fenchel-Nielsen coordinates imply that $\Fricke(\Sigma)\approx \R^{6g-6+3b}$ is connected, completing the proof
that $\Fricke(\Sigma)$ is a connected component of $\Hom(\Gamma,G)/G$.
(See, for example, Buser~\cite{MR1183224}, Abikoff~\cite{MR590044},
Thurston~\cite{MR1435975},\S4.6 or Wolpert~\cite{MR2641916} for details.) 

Tangent vectors to  $\Fricke(\Sigma)$ identifies with affine deformations of Fuchsian representations
of $\pi_1(\Sigma) \cong \Gamma_0 \subset  G$.
An {\em affine deformation\/} of $\Gamma_0$ consists of a lift of $\Gamma_0\subset\SOto$ to:
\begin{equation*}
\xymatrix{
& \Iso \ar[d]^{L}  \\
\Gamma_0 \ar@{-->}[ur]^{\rho} 
\ar@{^{(}->}[r] & \SOto}
\end{equation*}
where $\Eto$ is $3$-dimensional {\em Minkowski space,\/} that is, a simply connected geodesically
complete flat Lorentzian $3$-manifold, and $L$ is projection onto the linear part. 
Affine deformations of Fuchsian groups arise as holonomy groups
of {\em Margulis spacetimes,\/} complete flat Lorentz $3$-manifolds with free fundamental group.

 %As noted in Goldman-Margulis~
% \cite{{MR1796129}}, 
 Affine deformations correspond to 
 infinitesimal deformations of the hyperbolic structure on the quotient surface
 $\Sigma_0 := \Ht /\Gamma_0$, using the identification of $\Iso$ with the 
{\em tangent bundle  Lie group\/} as in \eqref{eq:TangentBundleLieGroup}.
The infinitesimal deformation theory was developed by Weil~\cite{MR0169956} 
(see also Raghunathan~\cite{MR0507234}, \S 6),
whereby the tangent space to $\Fricke(\Sigma)$ identifies with the cohomology group of
$\pi_1(\Sigma)$ with coefficients in the composition of the holonomy representation with the 
adjoint representation $\PSLtR \longrightarrow \mathsf{Aut}(\mathfrak{so}(2,1))$. 

The tangent space to the Fricke %Teichm\"uller 
space at $X$ identifies with $H^1(\Gamma,\mathfrak{g})$, 
where $\Gamma$ is the image of the holonomy representation for $X$ and 
$\mathfrak{g}$ is the Lie algebra of the group of isometries of the hyperbolic plane.  
Since $\mathfrak{g}\cong\R^{2,1}$, a tangent vector $V$ corresponds to an {\em affine deformation} of $\Gamma$, 
that is, a group of affine isometries of the Minkowski spacetime $\R^{2,1}$.  
Given $\gamma\in\Gamma$, 
the affine deformation of $\gamma$ acts by translations
on a unique invariant spacelike line $\mathsf{Axis}(\gamma)$.
The signed Lorentzian distance under which this translation displaces is called the 
{\em Margulis invariant\/} $\alpha(\gamma)$ of the affine deformation.

A {\em deformation\/} of hyperbolic surfaces $X_t$ is a one-parameter
family of marked hyperbolic surfaces of fixed topology, varying smoothly in $t\in [a,b]$,
where $[a,b]\in\R$ is a closed interval.
The space of equivalence classes of marked hyperbolic surfaces forms a smooth
manifold, corresponding to a smooth submanifold of the space of
representations $\pi_1(X)\longrightarrow \PSLtR$. 
Thus a deformation $X_t$ corresponds to a smooth path of holonomy 
representations $\rho_t$ in the top-dimensional stratum of $\Hom\big(\pi_1(X),\PSLtR\big)$. 
{\em Infinitesimal deformations\/} of $X$ correspond to tangent vectors in this space,
and also to affine deformations defined above.
%Under this correspondence, 
The Margulis invariant equals the derivative of the geodesic length function 
$\ell_{X_t}(\gamma)$ at $t=0$ 
(Goldman-Margulis~\cite{MR1796129}).
%(Compare Goldman-Margulis~\cite{MR1796129} for more details.)

%Those affine deformations which act {\em properly,\/} 
%that is, for which the quotient $\Eto/\Gamma$ is a (Hausdorff) manifold, correspond to 
%infinitesimal deformations whose geodesic length functions are {\em uniformly increasing}\,:
%%that is, their derivatives\: 
%% \[
%% \alpha(\gamma) := \frac{d}{dt} \ell_t(\gamma) \]
%% satisfy
%\begin{equation}\label{eq:alphaLowerBound}
%\kappa_2\ell(\gamma) \le  \alpha(\gamma) 
%\end{equation}
%for some $0 < \kappa_2$ chosen for all $\gamma\in \Gamma$
%(\cite{MR2600870}, \cite{GLMM} and \cite{DGK1}).

%Let $X_t$ be a smooth path of hyperbolic structures with $X_0=X$; 
%explicitly, $X_t$ is a path in Fricke-Teichm\"uller space through $X$.  
%The path corresponds to a path of holonomy representations $\rho_t$, with $\rho_0$ being the identity.  
%%(Although the holonomy representation is determined only up to conjugation,
%%this is irrelevant here.)
%For any closed geodesic $\gamma\subset X$, set\,:
%%
%\begin{equation*}
%\ell_\gamma(X_t):=\ell_{\rho_t(\gamma)}
%\end{equation*}
%%The illustration in \cite{CDG} depicts the lines defined by Margulis functionals in the case of a proper affine
%%deformation of a $1$-holed torus
%%
%

By \cite{MR2600870}, the $\R$-valued function 
%\begin{align*}
%\Gamma_0 &\longrightarrow \R \\
%\gamma &\longmapsto \frac{\alpha(\gamma)}{\ell(\gamma)}
%\end{align*}
$\gamma \mapsto \alpha(\gamma)/\ell(\gamma)$ on $\Gamma_0$ 
extends to a continuous function on the compact set of {\em geodesic currents\/}
when $\Gamma_0$ is convex cocompact, and is  therefore bounded.
Thus
\begin{equation}\label{eq:alphaUpperBound}
\vert \alpha(\gamma) \vert   \le \kappa\ell(\gamma) \end{equation}
for some $0 < \kappa$ and all $\gamma\in \Gamma$.
Inequality \eqref{eq:alphaUpperBound}
also holds when $\Gamma_0$ is only assumed to finitely generated;
for the elementary argument (which uses trace identities in $\mathsf{SL}(2)$), 
see \cite{GLMM}.

\section{Length identities for simple geodesics}

Let $X$ be a complete hyperbolic surface with one end,
which we denote $\partial$.   %$\partial X$.
The length of a closed geodesic $\gamma$ in $X$ will be written 
$\ell_X(\gamma)$, or simply $\ell(\gamma)$, when the context is clear. 
When $\gamma$ is peripheral, then defining $\ell_X(\gamma)$ to be zero
makes the function $X \longmapsto \ell_X(\gamma)$ continuous on the
deformation space.

% \ell_Xi:= \ell_{\beta_i}(X).
%\ell_Xi:= \ell_X(\beta_i)
% \end{equation*}

%Supppose that $X$ is an orientable hyperbolic surface with a single boundary 
% component $\beta_1$.  
The end of $X$ lies in various three-holed spheres %pairs of pants 
such that the other two ends correspond to simple closed geodesics 
$\gamma_1,\gamma_2 \subset X$.
(These curves coincide when $g=1$.) % has genus one).  
The set $\FO$ of isotopy classes of such subsurfaces identifies with 
the set of subsets $\gammas$ 
of isotopy classes of simple closed curves
%Let $\FO$ denote the set of unordered pairs of isotopy classes of simple closed curves $\gamma_j,\gamma_k$ 
such that $\partial,\gamma_j,\gamma_k$ cobound a %pair of pants in $X$.
three-holed sphere in $X$. Set\,:
\begin{equation}\label{eq:GenMcShaneR}
\OB(X_t):=\sum_{\gammas\in \FO} 
\DD\big(\ell_{X_t}(\partial), \ell_{X_t}(\gamma_1),\ell_{X_t}(\gamma_2) \big)
\end{equation}
where $\DD$ is the continuous function\,:
\begin{equation}\label{eq:DFunction}
\DD(x,y,z)  := 
2 \log \bigg( \frac{e^{x/2} + e^{(y+z)/2}} {e^{-x/2} + e^{(y+z)/2}}\bigg)  
\end{equation}

%\begin{equation}\label{eq:DFunction}
%\DD(x,y,z)  := 
%\begin{cases} 
%2 \log \bigg( \frac{e^{x/2} + e^{(y+z)/2}} {e^{-x/2} + e^{(y+z)/2}}\bigg) &\text{~if~} x > 0 \\
%(1 + e^{(y+z)/2})^{-1} &\text{~if~} x = 0. \end{cases}\end{equation}
%
%See Mirzakhani~\cite{MR2264808} for the geometric meaning of $\DD$.
%and the fact that \eqref{eq:GenMcShaneR} converges absolutely.
\noindent
The generalized McShane identity for one-ended orientable $X$  is\,: 

\begin{equation*}
\ell_X(\partial) =\OB(X).
\end{equation*}

We briefly discuss the convergence of this series, 
as the ideas will be crucially used when we differentiate it.
The key idea is that while the number of simple closed geodesics grows polynomially,
the summands in these series decay exponentially. 

Break $\FO$ into a countable disjoint union of finite sets $\FN{N}$ as follows.
For every nonnegative integer $N$, 
\begin{equation*}
\FN{N}:=\Big\{\gammas\in\FO~\Big|~N\leq 
\big(\ell(\gamma_1) + \ell(\gamma_2)\big)
< N+1\Big\}
\end{equation*}
is finite with cardinality
\begin{equation}\label{eq:Mirzakhani}
\mid\FN{N}\mid\leq m(N+1)^{6g-4}
\end{equation}
\big(Mirzakhani~\cite{MR2415399},
Proposition~3.1 (3.15) and  Rivin~\cite{MR1866856}(1)\big). 
Compare also Rivin~\cite{MR2174101} and Wolpert~\cite{MR2641916} for further discussion.

%Since the series in \eqref{eq:GenMcShaneR} converges absolutely, %we can rearrange its terms as follows\,:
%%
%\begin{equation*}
%\OB(X_t)=\sum_{N=0}^\infty f_N(t)  \end{equation*} where $f_N(t)$ denotes the finite sum:
%\begin{equation*}
%f_N(t) : =\sum_{\gammas\in\FN{N}}\DD\big(\ell_{X_t}(\partial X), \ell_{X_t}(\gamma_1),\ell_{X_t}(\gamma_2) \big).
%\end{equation*}

%
%The constant $m = m(g,b,Q)$ depends only on the genus $g$ and number $b$ of boundary components of $X_t$, which are constant, as well as a quantity  $Q$ defined as follows.
The constant $m = m(g,Q)$ depends only on the genus $g$ and  a quantity  $Q$ defined as follows.
Choose $\epsilon>0$ sufficiently small so that closed geodesics 
shorter than $\epsilon$ do not intersect, such as
\[ \epsilon <  2\sinh^{-1}(1) = 2\log(1+\sqrt{2}); \] 
compare Buser~\cite{MR1183224}, \S4.
The Margulis lemma and the collar lemma imply only finitely many such geodesics exist.
Explicitly, any set of disjoint closed geodesics extends to a pants decomposition.
Let $b$ denote the number of ends of $X$. 
Since the number of pants equals $-\chi$ where $\chi = 2-2g -b$ is the Euler characteristic,
at most $3g-3 + b$  closed geodesics are shorter than $\epsilon$.

Let $Q$ be the product of the inverses of the lengths of these geodesics.
%the product of the inverses of the lengths of closed geodesics which are shorter than $\epsilon$, where $\epsilon>0$ is chosen to be sufficiently small so that these geodesics do not intersect.  
By taking $t$ in an interval $[a,b]$ containing $0$, the continuity of $Q$ implies
we can choose $m$ so that inequality \eqref{eq:Mirzakhani} holds for all 
$t\in [a,b]$ although the finite subsets $\FN{N}$ may vary with $t$.

To prove that the series \eqref{eq:GenMcShaneR} converges absolutely, 
rearrange its terms as follows:
\begin{equation*}
f(t) := \OB(X_t)=\sum_{N=0}^\infty f_N(t)  \end{equation*} where $f_N(t)$ denotes the finite sum of positive terms:
\begin{equation}\label{eq:Deffn}
f_N(t) : =\sum_{\gammas\in\FN{N}}\DD\big(\ell_{X_t}(\partial), \ell_{X_t}(\gamma_1),\ell_{X_t}(\gamma_2) \big).
\end{equation}

\begin{lem}
\begin{equation}\label{eq:BoundingDD}
\vert \DD(x,y,z)\vert  \le   4 \sinh(x/2) \exp\big(\!-(y+z)/2  \big). \end{equation}
\end{lem}
\begin{proof}
Let $X = e^{x/2} > 0 $ and $Y= e^{(y+z)/2}> 0$. 
Write
\[ \DD(x,y,z) = 2 \log \bigg(
\frac{X+Y}{X^{-1}+Y}
\bigg) \]
where
\[ \frac{X+Y}{X^{-1}+Y}  \ =\  1 + \frac{X-X^{-1}}{X^{-1}+Y}. \]
Using the estimate $0 \le \log(1+u) \le u$ for $u\ge 0$, and
taking $u = \big(X-X^{-1}\big)/\big(X^{-1}+Y\big)$,
\begin{align*}
\DD(x,y,z) =
2 \log \bigg( \frac{X+Y}{X^{-1}+Y}\bigg)   & \le 
2 \ \frac{X-X^{-1}}{X^{-1}+Y}  \\
& < 2 \big(X-X^{-1}\big) Y^{-1} \\
&  = 4 \sinh(x/2) e^{-(y+z)/2}. \end{align*}
\end{proof}

\noindent
Combine \eqref{eq:Mirzakhani} with \eqref{eq:BoundingDD},
and the definition \eqref{eq:Deffn} of $f_N$, obtaining:
\[ \vert f_N(t) \vert \le A_N :=  4m \sinh\big(\ell_X(\partial)/2\big) (N+1)^{6g-4}  e^{-N/2}. \]
Since $\sum A_N < \infty$, the series \eqref{eq:GenMcShaneR} for the generalized McShane identity
converges uniformly and absolutely.

\section{Differentiating length identities}\label{sec:Differentiation}
Now differentiate the generalized McShane identity
to establish uniform absolute convergence of $\sum f_N(t)$.
%To simplify notation, define two functions $H(u,v)$ and $K(u,v)$:
%\begin{align*}
%H(u,v) &:= \frac1{1 + e^{(u+v)/2}}+ \frac1{1 + e^{(u-v)/2}} \\
%K(u,v) &:= \frac1{1 + e^{(u+v)/2}}- \frac1{1 + e^{(u-v)/2}}  %\\ 
%=  \frac{-\sinh(v/2)}{\cosh(u/2) + \cosh(v/2)}
%\end{align*}
%\big(as in Mirzakhani~\cite{MR2264808} for $H(u,v)$\big).
%For fixed $y$, the functions
%\begin{align*}
%x & \longmapsto \frac1{1 + e^{x+y}} \\
%x & \longmapsto H(x,y) \\
%x & \longmapsto K(x,y) \end{align*}
%are strictly decreasing.
%
%\end{lem}
\noindent
%Express the partial derivatives of $\DD(x,y,z)$ in terms of $H$ and $K$\,:
The partial derivatives 
\begin{align}\label{eq:DerivsHandK}
\frac{\partial\DD(x,y,z)}{\partial x} &= H(y+z, x), \\
\frac{\partial\DD(x,y,z)}{\partial y} = 
\frac{\partial\DD(x,y,z)}{\partial z} & = K(-y-z, x) \notag \end{align}
%\begin{align}\label{eq:DerivsHandK}
%\frac{\partial\DD}{\partial x}& = H(y+z, x)   \notag \\
%%\frac{\partial\DD}{\partial y}& = K(-y-z, x)   \notag\\
%%\frac{\partial\DD}{\partial z}& = K(-y-z, x). \end{align}
%\frac{\partial\DD}{\partial y} = 
%\frac{\partial\DD}{\partial z}& = K(-y-z, x). \end{align}
%&& =   \frac1{1 + e^{(x+y+z)/2}} + \frac1{1 + e^{(-x+y+z)/2}} > 0  \notag \\
%\frac{\partial\DD}{\partial y}  &= \quad
%\frac{\partial\DD}{\partial z}  && = 
%K(y+z,x) := 
%\frac1{1 + e^{(x-y-z)/2}}  - \frac1{1 + e^{(-x+y+z)/2}} > 0.
%\end{alignat*}
%\begin{alignat*}{2}%\label{eq:DerivativeDD1}
% \frac{\partial\DD}{\partial x}& = H(y+z, x) 
%&& =   \frac1{1 + e^{(x+y+z)/2}} + \frac1{1 + e^{(-x+y+z)/2}} > 0  \notag \\
%\frac{\partial\DD}{\partial y}  &= \quad
%\frac{\partial\DD}{\partial z}  && = 
%K(y+z,x) := 
%\frac1{1 + e^{(x-y-z)/2}}  - \frac1{1 + e^{(-x+y+z)/2}} > 0.
%\end{alignat*}
%\begin{equation}\label{eq:DerivativeDD1}
% \frac{\partial\DD}{\partial x} = H(y+z, x) 
%=  \frac1{1 + e^{(x+y+z)/2}} + \frac1{1 + e^{(-x+y+z)/2}} 
%\end{equation}
%and 
%\begin{equation}\label{eq:DerivativeDD2}
%\frac{\partial\DD}{\partial y}  = 
%\frac{\partial\DD}{\partial z}  = 
%\frac1{1 + e^{(x-y-z)/2}}  - \frac1{1 + e^{(-x-y-z)/2}} > 0 
%\end{equation}
decay exponentially as well:
%\begin{align}\label{eq:BoundingDerivDD}
%\Big| \frac{\partial\DD}{\partial x}\Big| &  < 2 \exp\big(-(y+z)/2\big)  \notag \\
%\Big|\frac{\partial\DD}{\partial y}\Big| =  \Big|\frac{\partial\DD}{\partial z}\Big| 
%&  <  2\sinh(\vert x\vert/2)\exp\big(-(y+z)/2\big)  \end{align}
\begin{equation}\label{eq:BoundingDerivDD}
\Big|\frac{\partial\DD}{\partial x}\Big|, 
\Big|\frac{\partial\DD}{\partial y}\Big| =  \Big|\frac{\partial\DD}{\partial z}\Big| 
 <  2\cosh(\vert x\vert/2)\exp\big(\!-(y+z)/2\big)  \end{equation}
\begin{lem}\label{lem:Derivfn}
The derivative $f_N'(t)$  is the sum over $\gammas\in\FN{N}$ of terms\,:
\begin{align}\label{eq:FormalDerivative2}
\frac{d}{dt}% \bigg|_{t=0} 
&  \DD\big(\ell_{X_t}(\partial), \ell_{X_t}(\gamma_1),\ell_{X_t}(\gamma_2) \big)    =\notag\\ 
&\qquad H\big(\ell_X(\gamma_1)+\ell_X(\gamma_2),\ell_X(\partial)\big)\ \alpha(\partial) \quad   +  \qquad  \notag \\ 
&\qquad\qquad K\big(\ell_X(\gamma_1)+\ell_X(\gamma_2),\ell_X(\partial)\big)\  %\alpha(\partial X) \quad   
\big(\alpha(\gamma_1) + \alpha(\gamma_2) \big) \notag
\end{align}
\end{lem}

%
%&\qquad\qquad K\big(\ell_X(\gamma_1)+\ell_X(\gamma_2),\ell_X(\partial)\big)\  %\alpha(\partial X) \quad   
%\big(\alpha(\gamma_1) + \alpha(\gamma_2) \big)  + \notag \\
%&\qquad\qquad \frac{\partial K}{\partial u}\big(\ell_X(\gamma_1) +\ell_X(\gamma_2),\ell_X(\partial)\big)
%\alpha_1(\partial)^2  + \notag \\
%&\qquad\qquad \frac{\partial K}{\partial v}\big(\ell_X(\gamma_1) +\ell_X(\gamma_2),\ell_X(\partial)\big)
%\alpha_1(\partial \big(\alpha_1(\gamma_1) + \alpha_1(\gamma_2) \big) \notag \\
%&\qquad\qquad \frac{\partial H}{\partial u}\big(\ell_X(\gamma_1) +\ell_X(\gamma_2),\ell_X(\partial)\big)
%\alpha_1(\partial \big(\alpha_1(\gamma_1) + \alpha_1(\gamma_2) \big) \notag \\
%&\qquad\qquad \frac{\partial H}{\partial v}\big(\ell_X(\gamma_1) +\ell_X(\gamma_2),\ell_X(\partial)\big)
%(\partial \big(\alpha_1(\gamma_1) + \alpha_1(\gamma_2) \big)^2  \notag 
%\end{align}
%\end{lem}&\qquad\qquad K\big(\ell_X(\gamma_1)+\ell_X(\gamma_2),\ell_X(\partial)\big)\  %\alpha(\partial X) \quad   

\begin{proof} Apply \eqref{eq:DerivsHandK} and the chain rule. \end{proof}
%By \eqref{eq:DerivsHandK}, the derivative $f_N'(t)$  is the sum over $\gammas\in\FN{N}$ of terms\,:
%\begin{align}\label{eq:FormalDerivative1}
%\frac{d}{dt}
%&  \DD\big(\ell_{X_t}(\partial), \ell_{X_t}(\gamma_1),\ell_{X_t}(\gamma_2) \big)    =\notag\\ 
%&\qquad H\big(\ell_X(\gamma_1)+\ell_X(\gamma_2),\ell_X(\partial)\big)\ \alpha(\partial) \quad   +  \qquad  \notag \\ 
%&\qquad\qquad K\big(\ell_X(\gamma_1)+\ell_X(\gamma_2),\ell_X(\partial)\big)\  
%\big(\alpha(\gamma_1) + \alpha(\gamma_2) \big) \notag
%\end{align}

\noindent
By \eqref{eq:BoundingDerivDD}, the coefficient 
$H\big(\ell_{X_t}(\gamma_1)+\ell_{X_t}(\gamma_2),\ell_{X_t}(\partial) \big)$ of 
$\alpha(\partial)$ and the coefficient 
$K\big(\ell_{X_t}(\gamma_1)+\ell_{X_t}(\gamma_2),\ell_{X_t}(\partial) \big)$ of 
$\alpha(\gamma_1)+\alpha(\gamma_2)$ 
%in 
%\[ \frac{d}{dt}% \bigg|_{t=0}
%  \DD\big(\ell_{X_t} (\partial), \ell_{X_t}(\gamma_1),\ell_{X_t}(\gamma_2) \big) \]
above % 
are each bounded by 
%\[ 2\exp\big(-(\ell_{X}(\gamma_1)+\ell_{X}(\gamma_2)/2\big),\]
\[ 2\cosh\big(\ell_{X_t}(\partial)/2\big)\exp\big(-(\ell_{X}(\gamma_1)+\ell_{X}(\gamma_2)/2\big). \]
%which, by  \eqref{eq:BoundingDerivDD}, bounds the coefficient of $\alpha(\gamma )$ as well.
%bounded by $2\exp\big(-(\ell_{X}(\gamma_1)+\ell_{X}(\gamma_2)/2\big)$ as well.
The contributions from the first term for $\gammas\in\FN{N}$ are bounded by
%\[ 2\exp(-N/2) \cosh\big(\ell_{X_t}(\partial)/2\big) \vert\alpha(\partial)\vert \]
\[ 2 \cosh\big(\ell_{X_t}(\partial)/2\big) \vert\alpha(\partial)\vert\  e^{-N/2}\]
and the contributions from the second terms are bounded by
\begin{align*} 
4 \cosh\big(\ell_{X_t}(\partial)/2\big) &\vert\alpha(\gamma_1) + \alpha(\gamma_2) \vert  e^{-N/2}\\
& \le  8\cosh\big(\ell_{X_t}(\partial)/2\big) \kappa \ N e^{-N/2}
\end{align*}
%\begin{align*} 
%4\exp(-N/2) \cosh\big(\ell_{X_t}(\partial)/2\big) &\vert\alpha(\gamma_1) + \alpha(\gamma_2) \vert  \\
%& \le  8\exp(-N/2) \cosh\big(\ell_{X_t}(\partial)/2\big) \kappa N
%\end{align*}
using  the linear bound for the Margulis invariant \eqref{eq:alphaUpperBound}.
Adding these contributions,  $\vert f_N'(t)\vert$ is bounded by 
%\[ B_N := 4\big(2\kappa N+ \vert \alpha(\partial)\vert\big)  m (N+1)^{6g-6+2b}  \exp(-N/2)
%\cosh\big(\ell_{X_t}(\partial)/2\big), \]
%\[begin{align}\label{eq:M}
\[ M_N :=   4 m \cosh\big(\ell_{X_t}(\partial)/2\big)\ 
\big(2\kappa N+ \vert \alpha(\partial)\vert\big) (N+1)^{6g-4}  %\exp(-N/2). \] %\notag \end{align}
e^{-N/2}. \] %\notag \end{align}

\begin{prop}\label{prop:IdentityDerivedR} 
Let $X_t$ be a smooth path of marked complete orientable one-ended hyperbolic surfaces
tangent to
%with $X_0=X$ and tangent vector 
%$\frac{d}{dt}\bigg|_{t=0}X_t$ given by 
an affine deformation with Margulis invariant $\alpha$.
Then\,:
%\begin{align*}
%\alpha(\partial) & =  \frac{d}{dt}\bigg|_{t=0} \OB(X_t) \\
%& = 
\begin{align*}
\sum_{\gammas\in \FO} 
& \Bigg( H\big(\ell_X(\gamma_1)+\ell_X(\gamma_2),\ell_X(\partial)\big)\; \alpha(\partial) \\  
& + 
K\big(\ell_X(\gamma_1)+\ell_X(\gamma_2),\ell_X(\partial)\big) 
\big( \alpha(\gamma_1) + \alpha(\gamma_2) \Bigg). \end{align*}
%Furthermore the above series converges absolutely.
converges absolutely to $\alpha(\partial)$.
\end{prop}
\begin{proof}
Evidently $\sum M_N < \infty$, and $\vert f_N'(t)\vert < M_N$.
Then 
% MR0166310 second edition
$\sum f_N'(t)$ converges uniformly and absolutely 
to $f'(t)$
(see, for example, Rudin~\cite{MR0385023}).
\end{proof}
\noindent
Theorem~\ref{thm:OneEnd} follows from Proposition~\ref{prop:IdentityDerivedR}.
Corollary~\ref{cor:IdentityDerived} now follows from Proposition~\ref{prop:IdentityDerivedR}
and the following lemma:
\begin{lem}{(Mirzakhani)} \label{lem:Mirzakhani}
\[
\sum_{\gammas\in \FO}  H\big(\ell_X(\gamma_1)+\ell_X(\gamma_2),\ell_X(\partial)\big) > 1 \]
\end{lem}
%\begin{thm}\label{thm:IdentityDerivedR}
%Let $X_t$ be a deformation of orientable hyperbolic surfaces with single boundary component such that, for
%all simple closed curves $\gamma$, the derivative of the length
%$\frac{d}{dt}\ell_{\gamma}(X_t)> 0$. Then $\frac{d}{dt} \ell_X(\partial) \neq 0$.
%\end{thm}
\begin{proof}
Differentiate $H(x,y)$ with respect to $y$, obtaining, for $x,y>0$\,:
\begin{equation}\label{eq:DerivH}
\frac{\partial H(x,y)}{\partial y}  =
\frac
{\exp\big((x-y)/2\big) (e^x-1) (e^y-1)}
{\big\{1 + \exp\big((x-y)/2\big)\big\}^2 
\big\{1 + \exp\big((x+y)/2\big)\big\}^2 } > 0
\end{equation}
%Now, fixing x,y (to be ell(gamma1) and ell(gamma2) respectively), look
%at the function
For fixed $x,y>0$ set\,:
\[f(L) :=  L H(x+y,L) - \DD(L,x,y) \]
where $L>0$.  Now $f(0)=0$, and\,: 
\[ f(L) = \int_0^L  f'(u) du = \int_0^Lu \frac{\partial H(x+y,u)}{\partial u} du> 0, \]
that is, 
$\DD(x,y,L)/L < H(x+y,L)$.

Let $L = \ell_X(\partial)$.  Then\,: 
\begin{align*}
1 & = 
\sum_{\gammas\in \FO} 
\DD\big(L, \ell_X(\gamma_1),\ell_X(\gamma_2) \big)/L \\
& \quad < 
\sum_{\gammas\in \FO} 
H\big( \ell_X(\gamma_1)+ \ell_X(\gamma_2), L)\big)
\end{align*}
as desired. \end{proof}
% Now we prove the main application in this paper.
\begin{proof} [Proof of Corollary~\ref{cor:IdentityDerived}]
Apply Proposition~\ref{prop:IdentityDerivedR}, obtaining:
\begin{align*} 
\bigg(1 - \sum_{\gammas\in \FO}  
& H\big( \ell_X(\gamma_1)+ \ell_X(\gamma_2), 
\ell_X(\partial) \big) \ \bigg) \alpha(\partial) \notag  \\
= 
\sum_{\gammas\in \FO} 
& K\big(\ell_X(\gamma_1)+\ell_X(\gamma_2),\ell_X(\partial)\big) 
\big(\alpha(\gamma_1) + \alpha(\gamma_2) \big)
\end{align*}
By Lemma~\ref{lem:Mirzakhani}, the coefficient on the left-hand side 
is negative. By \eqref{eq:DerivsHandK}, each coefficient on the right-hand side is negative also.
If all $\alpha(\gamma) > 0$ for interior curves, then $\alpha(\partial) > 0$ as desired.
\end{proof}

\bibliographystyle{amsplain}
%\bibliography{IdentityDerived}

\begin{thebibliography}{10}

\bibitem{MR1866854}
Herbert Abels, \emph{Properly discontinuous groups of affine transformations: a
  survey}, Geom. Dedicata \textbf{87} (2001), no.~1-3, 309--333. \MR{1866854}

\bibitem{MR590044}
William Abikoff, \emph{The real analytic theory of {T}eichm\"uller space},
  Lecture Notes in Mathematics, vol. 820, Springer, Berlin, 1980. \MR{590044}

\bibitem{MR3145000}
Mladen Bestvina, Ken Bromberg, Koji Fujiwara, and Juan Souto, \emph{Shearing
  coordinates and convexity of length functions on {T}eichm\"uller space},
  Amer. J. Math. \textbf{135} (2013), no.~6, 1449--1476. \MR{3145000}

\bibitem{MR1183224}
Peter Buser, \emph{Geometry and spectra of compact {R}iemann surfaces},
  Progress in Mathematics, vol. 106, Birkh\"auser Boston, Inc., Boston, MA,
  1992. \MR{1183224}

\bibitem{MR903850}
R.~D. Canary, D.~B.~A. Epstein, and P.~Green, \emph{Notes on notes of
  {T}hurston}, Analytical and geometric aspects of hyperbolic space
  ({C}oventry/{D}urham, 1984), London Math. Soc. Lecture Note Ser., vol. 111,
  Cambridge Univ. Press, Cambridge, 1987, pp.~3--92. \MR{903850}

\bibitem{MR3379819}
Virginie Charette and Todd~A. Drumm, \emph{Complete {L}orentzian 3-manifolds},
  Geometry, groups and dynamics, Contemp. Math., vol. 639, Amer. Math. Soc.,
  Providence, RI, 2015, pp.~43--72. \MR{3379819}

\bibitem{MR3180618}
Virginie Charette, Todd~A. Drumm, and William~M. Goldman, \emph{Finite-sided
  deformation spaces of complete affine 3-manifolds}, J. Topol. \textbf{7}
  (2014), no.~1, 225--246. \MR{3180618}

\bibitem{MR0227403}
Vicki Chuckrow, \emph{On {S}chottky groups with applications to kleinian
  groups}, Ann. of Math. (2) \textbf{88} (1968), 47--61. \MR{0227403}

\bibitem{DGK1}
Jeffrey Danciger, Fran{\c c}ois Gu{\'e}ritaud, and Fanny Kassel, \emph{Geometry
  and topology of complete lorentz spacetimes of constant curvature}, Ann. Sci.
  E.N.S. \textbf{49} (2016), no.~1, 1--56.

\bibitem{MR348055}
Jeffrey Danciger, Fran{\c{c}}ois Gu{\'e}ritaud, and Fanny Kassel,
  \emph{Margulis spacetimes via the arc complex}, Invent. Math. \textbf{204}
  (2016), no.~1, 133--193. \MR{3480555}

\bibitem{MR2987620}
Todd~A. Drumm, \emph{Lorentzian geometry}, Geometry, topology and dynamics of
  character varieties, Lect. Notes Ser. Inst. Math. Sci. Natl. Univ. Singap.,
  vol.~23, World Sci. Publ., Hackensack, NJ, 2012, pp.~247--280. \MR{2987620}

\bibitem{MR689763}
David Fried and William~M. Goldman, \emph{Three-dimensional affine
  crystallographic groups}, Adv. in Math. \textbf{47} (1983), no.~1, 1--49.
  \MR{689763}

\bibitem{MR884798}
W.~M. Goldman and J.~J. Millson, \emph{Local rigidity of discrete groups acting
  on complex hyperbolic space}, Invent. Math. \textbf{88} (1987), no.~3,
  495--520. \MR{884798}

\bibitem{MR957518}
William~M. Goldman, \emph{Geometric structures on manifolds and varieties of
  representations}, Geometry of group representations ({B}oulder, {CO}, 1987),
  Contemp. Math., vol.~74, Amer. Math. Soc., Providence, RI, 1988,
  pp.~169--198. \MR{957518}

\bibitem{GLMM}
William~M. Goldman, Fran\c cois Labourie, Gregory Margulis, and Yair Minsky,
  \emph{Geodesic laminations and proper affine actions}, (in preparation).

\bibitem{MR2600870}
William~M. Goldman, Fran{\c{c}}ois Labourie, and Gregory Margulis, \emph{Proper
  affine actions and geodesic flows of hyperbolic surfaces}, Ann. of Math. (2)
  \textbf{170} (2009), no.~3, 1051--1083. \MR{2600870 (2011b:30109)}

\bibitem{MR1796129}
William~M. Goldman and Gregory~A. Margulis, \emph{Flat {L}orentz 3-manifolds
  and cocompact {F}uchsian groups}, Crystallographic groups and their
  generalizations ({K}ortrijk, 1999), Contemp. Math., vol. 262, Amer. Math.
  Soc., Providence, RI, 2000, pp.~135--145. \MR{1796129 (2001m:53124)}

\bibitem{MR3379833}
Fran{\c{c}}ois Gu{\'e}ritaud, \emph{On {L}orentz spacetimes of constant
  curvature}, Geometry, groups and dynamics, Contemp. Math., vol. 639, Amer.
  Math. Soc., Providence, RI, 2015, pp.~253--269. \MR{3379833}

\bibitem{HTZ}
Hengnan Hu, Ser~Peow Tan, and Ying Zhang, \emph{A new identity for
  $\mathsf{SL}(2,\mathbb{C})$-characters of the once punctured torus group},
  Math. Res. Lett. \textbf{22} (2015), no.~2, 485--499.

\bibitem{MR1792613}
Michael Kapovich, \emph{Hyperbolic manifolds and discrete groups}, Progress in
  Mathematics, vol. 183, Birkh\"auser Boston, Inc., Boston, MA, 2001.
  \MR{1792613}

\bibitem{MR0223487}
D.~A. Ka{\v{z}}dan and G.~A. Margulis, \emph{A proof of {S}elberg's
  hypothesis}, Mat. Sb. (N.S.) \textbf{75 (117)} (1968), 163--168. \MR{0223487}

\bibitem{MR0218485}
J.-L. Koszul, \emph{Lectures on groups of transformations}, Notes by R. R.
  Simha and R. Sridharan. Tata Institute of Fundamental Research Lectures on
  Mathematics, No. 32, Tata Institute of Fundamental Research, Bombay, 1965.
  \MR{0218485}

\bibitem{MR2541705}
Fran{\c{c}}ois Labourie and Gregory McShane, \emph{Cross ratios and identities
  for higher {T}eichm\"uller-{T}hurston theory}, Duke Math. J. \textbf{149}
  (2009), no.~2, 279--345. \MR{2541705 (2010h:32008)}

\bibitem{MR722330}
G.~A. Margulis, \emph{Free completely discontinuous groups of affine
  transformations}, Dokl. Akad. Nauk SSSR \textbf{272} (1983), no.~4, 785--788.
  \MR{722330}

\bibitem{MR741860}
\bysame, \emph{Complete affine locally flat manifolds with a free fundamental
  group}, Zap. Nauchn. Sem. Leningrad. Otdel. Mat. Inst. Steklov. (LOMI)
  \textbf{134} (1984), 190--205, Automorphic functions and number theory, II.
  \MR{741860}

\bibitem{MR1625712}
Greg McShane, \emph{Simple geodesics and a series constant over {T}eichmuller
  space}, Invent. Math. \textbf{132} (1998), no.~3, 607--632. \MR{1625712
  (99i:32028)}

\bibitem{MR2346506}
\bysame, \emph{On the variation of a series on {T}eichm\"uller space}, Pacific
  J. Math. \textbf{231} (2007), no.~2, 461--479. \MR{2346506}

\bibitem{MR2264808}
Maryam Mirzakhani, \emph{Simple geodesics and {W}eil-{P}etersson volumes of
  moduli spaces of bordered {R}iemann surfaces}, Invent. Math. \textbf{167}
  (2007), no.~1, 179--222. \MR{2264808 (2007k:32016)}

\bibitem{MR2415399}
\bysame, \emph{Growth of the number of simple closed geodesics on hyperbolic
  surfaces}, Ann. of Math. (2) \textbf{168} (2008), no.~1, 97--125. \MR{2415399
  (2009c:32027)}

\bibitem{MR2399656}
Paul Norbury, \emph{Lengths of geodesics on non-orientable hyperbolic
  surfaces}, Geom. Dedicata \textbf{134} (2008), 153--176. \MR{2399656
  (2009b:32021)}

\bibitem{MR2587462}
Athanase Papadopoulos and Guillaume Th{\'e}ret, \emph{Shortening all the simple
  closed geodesics on surfaces with boundary}, Proc. Amer. Math. Soc.
  \textbf{138} (2010), no.~5, 1775--1784. \MR{2587462}

\bibitem{MR0507234}
M.~S. Raghunathan, \emph{Discrete subgroups of {L}ie groups}, Springer-Verlag,
  New York-Heidelberg, 1972, Ergebnisse der Mathematik und ihrer Grenzgebiete,
  Band 68. \MR{0507234}

\bibitem{MR1866856}
Igor Rivin, \emph{Simple curves on surfaces}, Geom. Dedicata \textbf{87}
  (2001), no.~1-3, 345--360. \MR{1866856 (2003c:57018)}

\bibitem{MR2174101}
\bysame, \emph{A simpler proof of {M}irzakhani's simple curve asymptotics},
  Geom. Dedicata \textbf{114} (2005), 229--235. \MR{2174101 (2006g:57033)}

\bibitem{MR0385023}
Walter Rudin, \emph{Principles of mathematical analysis}, third ed.,
  McGraw-Hill Book Co., New York-Auckland-D\"usseldorf, 1976, International
  Series in Pure and Applied Mathematics. \MR{0385023 (52 \#5893)}

\bibitem{MR2215456}
Ser~Peow Tan, Yan~Loi Wong, and Ying Zhang, \emph{Generalizations of
  {M}c{S}hane's identity to hyperbolic cone-surfaces}, J. Differential Geom.
  \textbf{72} (2006), no.~1, 73--112. \MR{2215456 (2007a:53087)}

\bibitem{MR1435975}
William~P. Thurston, \emph{Three-dimensional geometry and topology. {V}ol. 1},
  Princeton Mathematical Series, vol.~35, Princeton University Press,
  Princeton, NJ, 1997, Edited by Silvio Levy. \MR{1435975}

\bibitem{MR0137792}
Andr{\'e} Weil, \emph{On discrete subgroups of {L}ie groups}, Ann. of Math. (2)
  \textbf{72} (1960), 369--384. \MR{0137792}

\bibitem{MR0169956}
\bysame, \emph{Remarks on the cohomology of groups}, Ann. of Math. (2)
  \textbf{80} (1964), 149--157. \MR{0169956}

\bibitem{MR2641916}
Scott~A. Wolpert, \emph{Families of {R}iemann surfaces and {W}eil-{P}etersson
  geometry}, CBMS Regional Conference Series in Mathematics, vol. 113,
  Published for the Conference Board of the Mathematical Sciences, Washington,
  DC; by the American Mathematical Society, Providence, RI, 2010. \MR{2641916
  (2011c:32020)}

\end{thebibliography}
\providecommand{\bysame}{\leavevmode\hbox to3em{\hrulefill}\thinspace}
\providecommand{\MR}{\relax\ifhmode\unskip\space\fi MR }
\providecommand{\MRhref}[2]{%
  \href{http://www.ams.org/mathscinet-getitem?mr=#1}{#2}
}
\providecommand{\href}[2]{#2}

\end{document}